\documentclass{birkjour}
\usepackage{color}
\usepackage{setspace}

 \newtheorem{theorem}{Theorem}[section]
 
 \newtheorem{lemma}[theorem]{Lemma}
 \newtheorem{proposition}[theorem]{Proposition}
 \theoremstyle{definition}
 \newtheorem{definition}[theorem]{Definition}
 \theoremstyle{remark}
 \newtheorem{remark}[theorem]{Remark}
 \newtheorem*{ex}{Example}
 \numberwithin{equation}{section}

\begin{document}

%
%
%
%
%
%
%
%
%


\begin{center}
	{\Huge \noindent \textbf{Devaney Chaos on a Set-valued Map and Its Inverse Limit}}
	
	\vspace{.3cm}
	

	YINGCUI ZHAO$^{1}$ \quad LIDING WANG$^{2}$\quad NAN WANG$^{3}$
	
\emph{1 School of Computing, Dongguan University of Technology, Dongguan 523808, Guangdong Province, P. R. China.}

\emph{{2 School of Statistics and Data Science, Zhuhai College of Science and Technology, Zhuhai 519041, Guangdong Province, P. R. China.}}

\emph{{3 School of Mathematics, Jilin Universit, Changchun 13001, Jilin Province, P. R. China.}}

(email:\textcolor{blue}{1 zycchaos@126.com, 2 Corresponding author: wld0707@126.com, 3 wangnanchao@126.com})

\end{center}
	\hspace*{\fill} \\
{\flushleft \emph{Abstract.}\quad We study relationships between a set-valued map and its inverse limits about the notion of periodic point set, transitivity, sensitivity and Devaney chaos. Density of periodic point set of a set-valued map and its inverse limits implies each other. Sensitivity of a set-valued map and its inverse limits does not imply each other. Transitivity and Devaney chaos of generalized inverse limits implies the corresponding property of a set-valued map.}

\begin{flushleft}
	
\end{flushleft}

%
%
%
%



\begin{abstract}
We study relationships between a set-valued map and its inverse limits about the notion of periodic point set, transitivity, sensitivity and Devaney chaos. Density of periodic point set of a set-valued map and its inverse limits implies each other. Sensitivity of a set-valued map and its inverse limits does not imply each other. Transitivity and Devaney chaos of generalized inverse limits implies the corresponding property of a set-valued map.
\end{abstract}

\section{Introduction}\label{sec1}
A dynamical system $(X,f)$ is closely connected to the dynamical system of its inverse limit $(\underleftarrow{\lim}(X,f), \sigma_f)$. Chaos appears naturally from hyperbolicity conditions in smooth dynamics. And inverse limits are a useful tool to study the dynamical properties of smooth systems (see \cite{Williams1967}). What's more, some dynamical properties of $f$ can be interpreted as the topological structures of an inverse limits dynamical system. For discussions on the relationship between $(X,f)$ and the inverse limit dynamical system $(\underleftarrow{\lim}(X,f), \sigma_f)$ on various shadowing properties, mixing and chaos, see \cite{Gu2006,Lee2016,Liu2013}. In addition, L. Block \cite{Block1975} proved $\overline{P(\sigma_f)}=\underleftarrow{\lim}(\overline{P(f)},f)$, in which $P(f)$ is the periodic points set of $f$.
He and Liu \cite{He1999} proved that $(X, f)$ is transitive if and only if $(\underleftarrow{\lim}(X,f), \sigma_f)$ is transitive. Wu et. al. \cite{Wu2014} showed if $f$ is surjective, then $(X, f)$ is sensitive (respectively, Devaney chaotic) if and only if $(\underleftarrow{\lim}(X,f), \sigma_f)$ is sensitive (respectively, Devaney chaotic) .

In 2004, Mahavier\cite{Mahavier2004} introduced the concept of generalized inverse limits, or inverse limits with set-valued functions. This subject provides an new way to research multi-valued functions. What's more, this way does not lose any information after iteration. Actually generalized inverse limits may make it available that understand deeply the resulting topology and dynamics of the example and how the topology and dynamics are interacting.

Recently, the topic of generalized inverse limits is a greatly researched field of continuum theory \cite{Davies2021, Charatonik2021}. While most of the investigation has been on comprehending the topological structure of these spaces, some scholars have recently turned to studying the dynamical properties, for example specification property \cite{Raines2018}, topological entropy \cite{Erceg2018, Cordeiro2015}, chaos \cite{Judychaos}, etc.

Inspired by the above works on the dynamical properties of inverse limit dynamical systems, this paper studies chaotic properties via periodic points set, transitivity, sensitivity and Devaney chaos for the inverse limits of set-valued map. The specific layout of the present paper is
as follows. Some preliminaries and definitions are
introduced in the next section. Then we study the relationship between a set-valued map and its inverse limit about periodic points set and transitivity in Section 3, strong sensitivity and Devaney chaos in Section 4 and sensitivity and Devaney chaos in Section 5.
\section{Preliminaries}\label{sec2}
In this paper, we always suppose that $X$ is a compact metric space with a metric $d$ and $f:X\rightarrow X$ be a continuous map. Let $\mathbb{N}=\{0,1,2,\cdots\}$ and $\mathbb{Z^+}=\{1,2,3,\cdots\}$.
we say that $f$ is
\begin{enumerate}
	
	\item[(1)] \emph{(topologically) transitive}, if for any nonempty open sets $U,V\subset X$, there exists $n\in\mathbb{Z^+}$ such that $f^n(U)\bigcap V\neq\emptyset$.
	
	\item[(2)] \emph{exact}, if for any nonempty open set $U\subset X$, there exists $n\in\mathbb{Z^+}$ such that $f^n(U)=X$.
\end{enumerate}
The point $x\in X$ is said to be a \emph{periodic point} if there exists $n\in\mathbb{Z^+}$ such that $f^n(x)=x$. The \emph{period of $x$} is the smallest number $n\in\mathbb{Z^+}$ satisfying $f^n(x)=x$. We denote the set of all periodic points for $f$ by $P(f)$. For any $x\in X$ and any $\varepsilon>0$, let $B(x,\varepsilon)=\{y\in X:d(x,y)<\varepsilon\}$. We say that $f$ is \emph{sensitive}, if there exists $\delta>0$ such that for any $x\in X$ and any $\varepsilon>0$, there is $y\in B(x,\varepsilon)$ such that $d(f^n(x),f^n(y))>\delta$. $f$ is \emph{chaotic in the sense of Devaney (Devaney chaotic)} \cite{Devaney1989}, if it satisfies the following three conditions:

\begin{enumerate}
	\item[(1)] $f$ is transitive.
	
	\item[(2)] $\overline{P(f)}=X$, i.e., the set of all periodic points of $f$ is dense in $X$.
	
	\item[(3)] $f$ is sensitive.
\end{enumerate}
Note that Banks et al. \cite{Bank1992} proved that sensitivity follows from the other two conditions.

In this paper we consider the dynamics of a set-valued map and its inverse limit space. Let $2^X$ be the hyperspace of nonempty closed subsets of $X$ with Hausdorff metric.

A map $F:X\rightarrow 2^X$ is said to be \emph{upper semi-continuous at a point $x\in X$}, if for every open set $V\subset X$ containing $F(x)$, there exists an open set $U\subset X$ containing $x$ such that $F(t)\subset V$ for any $t\in U$. 
A map $F:X\rightarrow 2^X$ is said to be \emph{lower semi-continuous at a point $x\in X$}, if for any $y\in F(x)$ and every open set $V\subset X$ containing $y$, there exists an open set $U\subset X$ containing $x$ such that $F(t)\bigcap V\neq\emptyset$ for any $t\in U$. $F:X\rightarrow 2^X$ is said to be \emph{upper semi-continuous}(\emph{lower semi-continuous}), if it is upper semi-continuous(lower semi-continuous) at each point of $X$.

 
Throughout this paper, let $F:X\rightarrow 2^X$ be a upper semi-continuous function, and we call $F$ \emph{a set-valued map} and $(X, F)$ \emph{a set-valued dynamical system}. We define $F^k:X\rightarrow 2^X$, $k\in\mathbb{Z}^+$ as follows: for any $x\in X$ and any $k\in\mathbb{Z}^+$, $F^k(x)=\bigcup_{y\in F^{k-1}(x)}F(y)$.
\begin{definition}\label{defin1}
	 $F^{-1}:X\rightarrow 2^X$ is a set-valued map on $X$ and $F^{-1}(x)=\{y\in X:x\in F(y)\}$.
\end{definition}
\begin{proposition}\label{prop1}\cite{Li1998}
	The following statements are equivalent:
	\begin{enumerate}
		\item[(1)] $F$ is lower semi-continuous.
		
		\item[(2)] For every open subset $U\subset X$, $F^{-1}(U)$ is also an open subset in $X$.
	\end{enumerate}
\end{proposition}

The inverse limit space induced by $F$ is the space $$\underleftarrow{\lim} F=\{(x_0, x_1,\cdots)\in X^{\mathbb{N}}: x_{i}\in F(x_{i+1})\},$$ considered as a subspace of the Tychonoff product $X^{\mathbb{N}}$. We denote $(x_i)_{i=0}^\infty\in X^{\mathbb{N}}(x_i\in F(x_{i+1}), \forall i\geq 0)$ by $\overrightarrow{x}$, and $(x_i)_{i=0}^\infty\in X^{\mathbb{N}}(x_{i+1}\in F(x_{i}), \forall i\geq 0)$ by $\overleftarrow{x}$. Associated with the inverse limit spaces is a shift map $$\sigma(x_0,x_1,\cdots)=(x_1,x_2,\cdots).$$
Suppose that the diameter of $X$ is equal to $1$. For each $n\in\mathbb{N}$, we define a metric $\rho$ on $X^{\mathbb{N}}$ by $$\rho(x,y)=\sum_{i=0}^\infty\frac{d(x_i,y_i)}{2^i}.$$

We begin with a few simple extensions of definitions from the single-valued case. Note that orbits of $F$ are no longer uniquely determined by their initial condition.
\begin{definition}\label{defin2}\cite{Raines2018}
	An \emph{orbit} of a point $x\in X$ for $F$ is a sequence $\overleftarrow{x}$ such that $x_{i+1}\in F(x_i)$ and $x_0=x$.
\end{definition}
\section{Periodic points and transitivity}
\begin{definition}\label{defin4}\cite{Raines2018, Ansari2010}
Let $x\in X$ and $\overleftarrow{x}$ be an orbit of $x$. The orbit is said to be \emph{a periodic orbit}, if there exists $m\in\mathbb{Z^+}$ such that $x_i=x_{i+m}$ for any $i\geq 0$. The point $x$ is \emph{a periodic point}, if it has at least one periodic orbit. \emph{The period of $x$} is the smallest number $m\in\mathbb{Z^+}$ satisfying $x_i=x_{i+m}$, $\forall i\geq 0$. If $m=1$, then $x$ is said to be \emph{a fixed point}. We denote the set of all periodic points of $F$ by $P(F)$.
\end{definition}

By Definition \ref{defin2} and Definition \ref{defin4}, it is easy to see the orbits of $x\in X$ for $F:X\rightarrow 2^X$  no longer uniquely determined. And the orbit $\overleftarrow{x}$ is not necessarily periodic even if there exists $j\in\mathbb{N}$ such that $x=x_0=x_j$. Then we get the following lemma.
\begin{lemma}\label{lem1}
	$P(\sigma\mid_{\underleftarrow{\lim} F})=\underleftarrow{\lim} F\mid_{P(F)},$ where $F\mid_{P(F)}: P(F)\rightarrow 2^{P(F)}$ is defined as: $F\mid_{P(F)}(x)=F(x)\bigcap P(F), \forall x\in X.$
\end{lemma}

Let $F(U)=\bigcup_{x\in U}F(x)$, for any subset $U\subset X$. 
\begin{theorem}\label{thmpero}
	\begin{enumerate}
		\item[(1)] $\overline{P(F)}=X$ implies $\overline{P(\sigma\mid_{\underleftarrow{\lim} F})}=\underleftarrow{\lim} F$.
		
		\item[(2)] Let $F(X)=X$. Then $\overline{P(\sigma\mid_{\underleftarrow{\lim} F})}=\underleftarrow{\lim} F$ implies $\overline{P(F)}=X$.
	\end{enumerate}
\end{theorem}
\begin{proof}
	\begin{enumerate}
		\item[(1)] Suppose that $\overline{P(F)}=X$. Then for any nonempty open set $\mathcal{U}=U_0\times U_1\times U_2\times\cdots\subset\underleftarrow{\lim} F$, $U_i\bigcap P(F)\neq\emptyset$, $i\geq 0$. Hence, $\mathcal{U}\bigcap\underleftarrow{\lim} F\mid_{P(F)}\neq\emptyset$. So, $\overline{\underleftarrow{\lim} F\mid_{P(F)}}=\underleftarrow{\lim} F.$ By Lemma \ref{lem1}, $\overline{P(\sigma\mid_{\underleftarrow{\lim} F})}=\underleftarrow{\lim} F$.
		
		\item[(2)] Suppose that $F(X)=X$ and $\overline{P(\sigma\mid_{\underleftarrow{\lim} F})}=\underleftarrow{\lim} F$. Let $U$ be a nonempty open set of $X$. By $X=F(X)$, $U\times X\times X\times\cdots$ is a nonempty open set of $\underleftarrow{\lim} F$. Then, $U\times X\times X\times\cdots\bigcap\underleftarrow{\lim} F\mid_{P(F)}\neq\emptyset$. Hence, $U\bigcap P(F)\neq\emptyset$. So, $\overline{P(F)}=X$.
	\end{enumerate}
\end{proof}
\begin{theorem}\label{thm3}
	$\overline{P(F)}=X$ if and only if $\overline{P(F^{-1})}=X$.
\end{theorem}
\begin{proof}
	$\Rightarrow$: Let $U$ be a nonempty open set of $X$. By $\overline{P(F)}=X$, there exists $x\in U$ with a periodic orbit $\overleftarrow{x}$ for $F$ and the period $n$. Let $y_{kn}=x_{n-1}$, $y_{kn+1}=x_{n-2}$, $\cdots$, $y_{kn+n-2}=x_1$, $y_{kn+n-1}=x_0$, $\forall k\geq 0.$ Then $\overrightarrow{y}$ is a $n$-periodic orbit for $F^{-1}$. Let $z_i=y_{i+n-1}$, $\forall i\geq 0$. Then $\overrightarrow{z}$ is also a $n$-periodic orbit for $F^{-1}$ and $z_0=x_0=x$. Hence, $x\in P(F^{-1})$. So, $\overline{P(F^{-1})}=X$.
	
	$\Leftarrow$: It is similar.
\end{proof}
For any subset $A\subset X$, and any $\varepsilon>0$, let $B(A,\varepsilon)=\{x\in X: d(x, A)<\varepsilon\}$. Now we introduce the definition of transitivity for $F$.
\begin{definition}\label{defin3}
	The set-valued map $F$ is \emph{(topologically) transitive} if, for any non-empty open sets $U, V\subset X$, there is an $n\in\mathbb{Z}^+$ such that there is $x_0\in U$ with an orbit $\overleftarrow{x}$ satisfying $x_n\in V$.
\end{definition}
\begin{theorem}\label{thmtran}
	Let $F(X)=X$. If $\underleftarrow{\lim}$ is transitive via $\sigma$, then $F$ is transitive.
\end{theorem}
\begin{proof}
	Let $U_0$, $V_0\subset X$ be two non-empty open sets. Then we can get two non-empty open sets in $\underleftarrow{\lim} F$: $\mathcal{U}=U_0\times U_1\times U_2\times\cdots$, $\mathcal{V}=V_0\times V_1\times V_2\times\cdots$ with $U_{i}\subset F(U_{i+1})$ and $V_{i}\subset F(V_{i+1})$, $\forall i\in\mathbb{N}$. Hence, there exists $n\in\mathbb{N}$ such that $\sigma^n(\mathcal{U})\cap\mathcal{V}\neq\emptyset.$ Then, there exists $(x_i)_{i=0}^\infty\in\mathcal{U}$ such that $(x_n,x_{n+1},\cdots)\in\mathcal{V}$, in which $x_0\in F^n(x_n)$. Set the $(y_i)_{i=0}^\infty$ for $F$ with $y_n=x_0$ and $y_0=x_n$. Then, $y_0\in V_0$ and $y_n\in U_0$. So, $F$ is transitive.
\end{proof}
The following example shows the converse of Theorem \ref{thmtran} isn't true.
\begin{ex}\label{exatran}
	Consider the map$F:\{0,1\}\rightarrow 2^{\{0,1\}}$ defined as $F(0)=\{0,1\}$ and $F(1)=\{0\}$. Then $\underleftarrow{\lim} F=\{(x_i)_{i=0}^\infty:x_i=0$ implies $x_{i+1}=0~or~1$ and $x_i=1$ implies $x_{i+1}=0$, $\forall i\in\mathbb{N}\}$. It is easy to see $F$ is transitive. Set $\mathcal{U}=\{0\}\times\{1\}\times\{0\}\times\{1\}\times\cdots$, $\mathcal{V}=\{0\}\times\{0\}\times\{0\}\times\cdots$. Then, both $\mathcal{U}$ and $\mathcal{V}$ are non-empty open sets in $\underleftarrow{\lim} F$. However, for any $n\in\mathbb{N}$, $\sigma^n(\mathcal{U})\cap\mathcal{V}=\emptyset$. So, $\underleftarrow{\lim} F$ is not transitive via $\sigma$.
\end{ex}
\section{Strong sensitivity and Devaney chaos}
In this section, we will introduce the first kind of sensitivity for $F$, which is named strong sensitivity. And the corresponding strong Devaney chaos for $F$ will be given.
\begin{definition}\label{defin6}
	Let $(X, F)$ be a set-valued dynamical system. We say $F$ is \emph{strongly sensitive} if, there exists $\delta>0$ (strongly sensitive constant) such that for any $x\in X$ and any $\varepsilon>0$, there is $y\in B(x,\varepsilon)$ with an orbit $\overleftarrow{y}$ and $n\in\mathbb{Z^+}$ satisfying $$\inf_{x_n\in F^n(x)}d(x_n,y_n)>\delta.$$
\end{definition}
\begin{definition}
	Let $(X, F)$ be a set-valued dynamical system. We say $F$ is \emph{strongly Devaney chaotic}, if
	\begin{enumerate}
		\item[(1)] $F$ is transitive,
		
		\item[(2)] the periodic points set of $F$ is dense in $X$, i.e., $\overline{P(F)}=X$,
		
		\item[(3)] $F$ is strongly sensitive.
	\end{enumerate}
\end{definition}

\begin{proposition}\label{prop2}
	Let $(X, F)$ be a set-valued dynamical system. If $F(x)=X$ for some $x\in X$, then $F$ is not strongly sensitive.
\end{proposition}
\begin{proof}
	For any $\delta>0$, Let $x\in X$ with $F(x)=X$ and $\varepsilon=\frac{1}{3}>0$. Then, for any $y\in B(x,\varepsilon)$ and any $n>0$, $d(z,F^n(x))=0$, $\forall z\in F^n(y)$. So, $F$ is not strongly sensitive.
\end{proof}
The following example shows that there exists a set-valued map such that  both $\underleftarrow{\lim} F^{-1}$ and $\underleftarrow{\lim} F$ are sensitive via $\sigma$ but $F$ is not strongly sensitive.
\begin{ex}\label{exa2}
	Consider the set-valued map $F(x)=\{0,1\}$, $x\in\{0,1\}$.
	\begin{enumerate}
		\item[(1)] It is easy to see $\underleftarrow{\lim} F^{-1}=\underleftarrow{\lim} F=\{(x_0,x_1,x_2,\cdots):x_i\in\{0,1\},\forall i\geq 0\}$. So,  both $\underleftarrow{\lim} F^{-1}$ and $\underleftarrow{\lim} F$ are sensitive via $\sigma$.
		\item[(2)] By Proposition \ref{prop2}, $F$ is not strongly sensitive.
	\end{enumerate}
\end{ex}
\begin{remark}
	This example also shows Devaney chaos of $\underleftarrow{\lim} F^{-1}$ or $\underleftarrow{\lim} F$ via $\sigma$ can't imply strong Devaney chaos of $F$.
\end{remark}

The following example shows that there exists a set-valued map satisfying $F$ is strongly sensitive but $F^{-1}$ is not strongly sensitive. Also it shows that
\begin{enumerate}
	\item[(1)] strong sensitivity of $F$ can't imply the sensitivity of $\underleftarrow{\lim} F^{-1}$ via $\sigma$. 
	\item[(2)] strong Devaney chaos of $F$ can't imply the Devaney chaos of $\underleftarrow{\lim} F^{-1}$ via $\sigma$.
\end{enumerate}
\begin{ex}\label{exa3}
	Consider the set-valued map $F:[0,1]\rightarrow 2^{[0,1]}$ defined as \begin{equation*}
		F(x)=\begin{cases}
			\{2x,0\}, & 0\leq x\leq\frac{1}{2}, \\
			\{2-2x,0\}, & \frac{1}{2}<x\leq 1.
		\end{cases}
	\end{equation*}.
	\begin{enumerate}
		\item [(1)]Define $f:[0,1]\rightarrow [0,1]$ as
		\begin{equation*}
			f(x)=\begin{cases}
				2x, & 0\leq x\leq\frac{1}{2}, \\
				2-2x, & \frac{1}{2}<x\leq 1.
			\end{cases}
		\end{equation*}
		As we all know, $f$ is exact. Let $\delta=\frac{1}{8}$. Now we show that $\delta$ is a strongly sensitive constant for $F$.
		
		Let $x\in [0,1]$ and $\varepsilon>0$, then there exists $n>0$ such that $f^n(B(x,\varepsilon))=[0,1]$. Now we have to start splitting in the following situation.
		
		Case $1$. $f^n(x)+\delta\geq 1$. There exists $y\in B(x,\varepsilon)$ such that $f^n(y)\in(f^n(x)-2\delta,f^n(x)-\delta)$. Then, $f^n(y)>\delta$ and $f^n(y)<f^n(x)-\delta$. So, $d(f^n(y),F^n(x))>\delta.$
		
		Case $2$. $f^n(x)+\delta<1$. There exists $y\in B(x,\varepsilon)$ such that $f^n(y)\in(f^n(x)+\delta,1]$. Then, $f^n(y)>f^n(x)+\delta$. So, $d(f^n(y),F^n(x))>\delta.$
		
		By case $1$ and case $2$, $F$ is strongly sensitive.
		
		\item[(2)] 
	 \begin{equation*}
		F^{-1}(x)=\begin{cases}
			[0,1], & x=0, \\
			\{\frac{x}{2},1-\frac{x}{2}\}, & 0<x\leq 1.
		\end{cases}
	\end{equation*}
 Then by Proposition \ref{prop2}, $F^{-1}$ is not strongly sensitive.
		
		\item[(3)] Let $\delta>0$, $\overleftarrow{x}=(0,0,\cdots)\in\underleftarrow{\lim} F^{-1}$ and $\varepsilon=\frac{1}{3}$. Then for any $\overleftarrow{y}\in\underleftarrow{\lim} F^{-1}$,  $\rho(\overleftarrow{x},\overleftarrow{y})<\varepsilon$ impies $y_0=0$. Hence, $y_1=y_2=\cdots =0$. This means, for any $n>0$, $\rho(\sigma^n(\overleftarrow{x}),\sigma^n(\overleftarrow{y}))=0<\delta$. So, $\underleftarrow{\lim} F^{-1}$ is not sensitive via $\sigma$.
		
		\item[(4)] It is easy to see $F$ is strongly Devaney chaotic, but  $\underleftarrow{\lim} F^{-1}$ is not Devaney chaotic via  $\sigma$.
	\end{enumerate}
\end{ex}
The following two examples show the transitivity of $F$ and $\overline{P(F)}=X$ can't imply strong sensitivity of $F$.
\begin{ex}\label{exa4}
	Consider the set-valued map $F(0)=\{0,1\}$ and $F(1)=\{0\}$. It is easy to see $F$ is transitive and $P(F)=\{0,1\}$. By Proposition \ref{prop2}, $F$ is not strongly sensitive.
\end{ex}



\begin{ex}\label{exa8}
	Consider the set-valued map $F:[0,1]\rightarrow 2^{[0,1]}$: $F(x)=[0,1]$, $\forall x\in[0,1]$.
	\begin{enumerate}
		\item[(1)] Let $U,V$ be two nonempty open sets of $[0,1]$. Since $F(U)=[0,1]$, $F(U)\bigcap V\neq\emptyset$. So, $F$ is transitive.
		
		\item[(2)] for any $x\in[0,1]$, $x\in F(x)$. Then, $x\in P(F)$. So, $\overline{P(F)}=X$.
		
		\item[(3)] For any $\delta>0$, let $x=0$ and $\varepsilon=\frac{1}{2}$. Then for any $y\in B(x,\varepsilon)$, any $n>0$ and any $y_n\in F^n(y)$, $y_n\in[0,1]$. Hence, $d(y_n,F^n(x))=0$. So, $F$ is not  strongly sensitive.
	\end{enumerate}
\end{ex}

\section{Sensitivity and Devaney chaos}
As shown in the previous section, the relationships between a set-valued map and its inverse limit about neither strong sensitivity nor strong Devaney chaos is valid. Transitivity of $F$ and $\overline{P(F)}=X$ can't imply strong sensitivity of $F$. In thist section, we will give the sencond kind of sensitivity and Devaney chaos for $F$. Then we can get some different conslusions.
\begin{definition}\label{defin5}
	Let $(X, F)$ be a set-valued dynamical system. We say $F$ is \emph{sensitive}, if there exists $\delta>0$ (sensitive constant) such that for any $\varepsilon>0$, any $x\in X$ with its any orbit $(x_i)_{i=0}^\infty$, there is $y\in B(x,\varepsilon)$ with an orbit  $(y_i)_{i=0}^\infty$, and $n\in\mathbb{Z^+}$ satisfying $$d(x_n,y_n)>\delta.$$
\end{definition}

Note that $F$ is strongly sensitive implies its sensitivity.
The next two examples show the sensitivity of $F$ and $\underleftarrow{\lim} F$ via $\sigma$ doesn't imply each other.
\begin{ex}
	$F(x)=\{2x,3x\}$, $x\in[0,+\infty]$.
	\begin{itemize}
		\item[(1)] Let $\delta=1$, $x\in[0,+\infty]$ and $\varepsilon>0$. For any orbit of $x$ for $F$, denote it as $(x_0,x_1,x_2,\cdots,x_i,\cdots )$. Select $y_0=x_0+\varepsilon$ and its orbit $(y_0,y_1,y_2,\cdots,y_i,\cdots )$ for $F$ satisfying $\frac{y_{i+1}}{y_i}=\frac{x_{i+1}}{x_i},\forall i\geq 0$. Then for any $l>0$, $\mid y_l-x_l\mid>\varepsilon 2^l$. Hence there exists $n>0$ such that $\mid y_n-x_n\mid>\delta$. So, $F$ is sensitive.
		\item[(2)]Let $\delta>0$, $\overleftarrow{x}={0,\cdots,0,\cdots}$ and $\varepsilon=\frac{1}{3}$. Then $\overleftarrow{x}\in\underleftarrow{\lim} F$. For any $\overleftarrow{y}\in\underleftarrow{\lim} F$, if $\rho(\overleftarrow{x},\overleftarrow{y})<\varepsilon$, then $y_0=x_0=0$. Hence,  $\overleftarrow{y}={0,\cdots,0,\cdots}$. For any $n>0$, $\rho(\sigma^n(\overleftarrow{x}),\sigma^n(\overleftarrow{y}))<\delta$. So, $\underleftarrow{\lim} F$ is not sensitive via $\sigma$.
	\end{itemize}
\end{ex}

\begin{ex} Consider the set-valued map $F:[0,1]\rightarrow 2^{[0,1]}$ defined as
	\begin{equation*}
		F(x)=\begin{cases}
			\{x\}, & 0\leq x<\frac{1}{2}, \\
			[0,1], & x=\frac{1}{2},\\
			\{0,x,1\},& \frac{1}{2}<x<1,\\
			[0,1], & x=1.
		\end{cases}
	\end{equation*}
	\begin{itemize}
		\item [(1)]For any $x\in[0,1]$, $F^{-1}(x)$ is not a single point set, so $\underleftarrow{\lim} F$ is sensitive via $\sigma$.
		\item[(2)]Let $\delta>0$, $x=0$ and $\varepsilon=\frac{1}{4}$. Then for any $y\in[0,1]$ with $d(x,y)<\frac{1}{4}$ and any $n>0$, $d(x_n,y_n)<\delta$, $\forall x_n\in F^n(x)$, $\forall y_n\in F^n(y)$. So, $F$ is not sensitive.
	\end{itemize}
\end{ex}

Although the sensitivity of $F$ and the sensitivity of $\underleftarrow{\lim} F$ via  $\sigma$ doesn't imply each other, if the conditions are strengthened, implicative relations can be obtained.
\begin{theorem}
	Let $(X, F)$ be a set-valued dynamical system. Suppose that for any $x\in X$, its any orbit  $(x_i)_{i=0}^\infty$ is a periodic orbit. If $\underleftarrow{\lim} F$ is sensitive via $\sigma$, then $F$ is sensitive.
\end{theorem}
\begin{proof}
	Let $\delta$ be a sensitive constant of $\underleftarrow{\lim} F$ via  $\sigma$. Let $\varepsilon>0$, $x_0\in X$ and $$(x'_i)_{i=0}^\infty=(x_0,x_1,\cdots,x_M,x_1,\cdots,x_M,\cdots)(M\in\mathbb{Z^+})$$ be an orbit of $x_0$ for $F$. Let $$\overleftarrow{x}=(x_M,x_{M-1},\cdots,x_1,x_M,\cdots,x_1,\cdots).$$ Then there exists $$\overleftarrow{y}=(y_N,y_{N-1},\cdots,y_1,y_N,\cdots,y_1,\cdots)\in\underleftarrow{\lim} F$$ with $\rho(\overleftarrow{x},\overleftarrow{y})<\varepsilon$ and $n>0$ such that $\rho(\sigma^n(\overleftarrow{x}),\sigma^n(\overleftarrow{y}))>\delta$. Hence, there exists $n>0$, $y=y_N$ and its orbit $(y'_i)_{i=0}^\infty=(y_0,y_1,\cdots,y_N,y_1,\cdots,y_N,\cdots)$ such that $d(x_0,y)<\varepsilon$ and $d(x'_n,y'_n)>\delta$. So, $F$ is sensitive.
\end{proof}

\begin{definition}
	Let $(X, F)$ be a set-valued dynamical system. We say $F$ is Devaney chaotic, if
	\begin{enumerate}
		\item[(1)] $F$ is transitive,
		
		\item[(2)] the periodic points set of $F$ is dense in $X$, i.e., $\overline{P(F)}=X$,
		
		\item[(3)] $F$ is sensitive.
	\end{enumerate}
\end{definition}

\begin{theorem}\label{thm123}
	Let $(X, F)$ be a set-valued dynamical system. If $F$ is transitive and $\overline{P(F)}=X$, then $F$ is sensitive.
\end{theorem}
\begin{proof}
	For any $p_1, p_2\in P(F)$ with different periodic orbits $\overleftarrow{p_1}$ and $\overleftarrow{p_2}$, let $\delta_0=\rho(\overleftarrow{p_1},\overleftarrow{p_2})$. Then, either $d(x,\overleftarrow{p_1})<\frac{\delta_0}{2}$ or $d(x,\overleftarrow{p_2})<\frac{\delta_0}{2}$ holds. Hence, there exists $\delta_0>0$ such that for any $x\in X$ there is $p\in P(F)$ with a periodic orbit $\overleftarrow{p}$ satisfying $d(x,\overleftarrow{p})\geq\frac{\delta_0}{2}$.

	Let $\delta=\frac{\delta_0}{8}$. Now we show $\delta$ is a sensitive constant for $F$.
	
	For any given $x\in X$ and any $0<\varepsilon<\delta$. By $\overline{P(F)}=X$, there exists a periodic point $p$ with a periodic orbit $\overleftarrow{p}$ and the period $n$ such that $d(x,p)<\varepsilon$. Then there exists $q\in P(F)$ with a periodic orbit $\overleftarrow{q}$ such that $d(x,\overleftarrow{q})\geq 4\delta$. Set $$U=\bigcap_{i=0}^nF^{-i}(B(q_i,\delta)).$$ By Proposition \ref{prop1}, $U$ contains a nonempty open set of $X$. Since $F$ is transitive, there exist $y\in B(x,\varepsilon)$ and $k>0$ such that $y_k\in U$ for some $y_k\in F^k(y)$. Let $j=[\frac{k}{n}]+1$, then $k\leq jn\leq n+k.$ Hence, $0\leq jn-k\leq n$. Then, $$y_k\in F^{-(jn-k)}(B(q_{jn-k},\delta)).$$ This means, there is $x\in B(q_{jn-k},\delta)$ such that $x\in F^{jn-k}(y_k)$. Hence, there exists $y_{jn}\in F^{jn-k}(y_k)$ such that $y_{jn}\in B(q_{jn-k},\delta)$. Thus, we can get an orbit of $y$ for $F$, which is denoted by $\overleftarrow{y}$.
	
	Since $$d(x,q_{jn-k})\leq d(x,p)+d(p,q_{jn-k})\leq d(x,p)+d(p,y_{jn})+d(y_{jn},q_{jn-k}),$$ and $p_{jn}=p,$ $d(p_{jn},y_{jn})=d(p,y_{jn})\geq d(x,q_{jn-k})-d(y_{jn},q_{jn-k})-d(x,p)\geq d(x,\overleftarrow{q})-d(y_{jn},q_{jn-k})-d(x,p)\geq 4\delta-\delta-\delta=2\delta.$ For any $x_{jn}\in F^{jn}(x)$ $$d(p_{jn},x_{jn})+d(x_{jn},y_{jn})\geq d(y_{jn},p_{jn})\geq2\delta.$$ Then, either $d(p_{jn},x_{jn})>\delta$ or $d(x_{jn},y_{jn})>\delta$ holds. So, $F$ is sensitive.
\end{proof}
\begin{remark}
	For Theorem \ref{thm123}, we need $F$ is lower semi-continuous but not upper semi-continuous.
\end{remark}

\begin{theorem}\label{devaney}
	Let $(X, F)$ be a set-valued dynamical system. Suppose $F(X)=X$. If $\underleftarrow{\lim} F$ is Devaney chaotic via $\sigma$, then $F$ is Devaney chaotic.
\end{theorem}
\begin{proof}
By Theorem \ref{thmtran}, $F$ is transitive. By Theorem \ref{thmpero}, $\overline{P(F)}=X$.
\end{proof}

The following example shows that  Devaney chaos of  $F$ does not imply the Devaney chaos of $\underleftarrow{\lim} F$ via  $\sigma$.
\begin{ex}
	Consider the set-valued map $F$ from $\{0,1\}$ to $2^{\{0,1\}}$: $F(0)=\{0,1\}$ and $F(1)=\{0\}$. So, $F$ is  Devaney chaotic.
	\begin{itemize}
		\item [(1)]It is easy to see $F$ is transitive and $P(F)=\{0,1\}$. By Theorem \ref{thm123}, $F$ is Devaney chaoatic.
		\item[(2)]By Example \ref{exatran}, $\underleftarrow{\lim} F$ is not transitive. So, $\underleftarrow{\lim} F$ is not Devaney chaotic via $\sigma$.
	\end{itemize}
\end{ex}

\section*{Conflict of interest}
All authors declare no conflicts of interest in this paper.
\section*{Author Contributions Statement}
\textbf{Yingcui Zhao}: Conceptualization, Methodology, Investigation, Writing-Review \& Editing. \textbf{Lidong Wang}: Supervision, Provision of study materials. \textbf{Nan Wang}: Writing, Validation.

And all three authors reviewed the manuscript.

\end{document}